\documentclass[12pt,psfig]{article}
\usepackage{amsmath, amsthm, amssymb}
\usepackage{graphicx,epsfig}
\usepackage{amsfonts}
\usepackage{amscd}
\usepackage{mathrsfs}
\usepackage{enumerate}
\usepackage{latexsym}
\usepackage[all]{xy}

\newcommand{\NZ}{\mathbb{N}}

\newcommand{\eg}{{\rm e.g., }}

\newcommand{\CR}{{\rm CR}}
\newtheorem{thm}{Theorem}[section]

\newtheorem{lem}[thm]{Lemma}
\newtheorem{prob}[thm]{Problem}
\newtheorem{prop}[thm]{Proposition}

\theoremstyle{definition}
\newtheorem{defn}[thm]{Definition}

\newtheorem{rem}[thm]{Remark}

\numberwithin{equation}{section}

\def\be {\begin{equation}}
\def\ee {\end{equation}}
\def\ba {\begin{eqnarray}}
\def\ea {\end{eqnarray}}

\begin{document}
\baselineskip=18pt
\renewcommand {\thefootnote}{\dag}
\renewcommand {\thefootnote}{\ddag}
\renewcommand {\thefootnote}{ }

\pagestyle{empty}

\begin{center}
                \leftline{}
                \vspace{-0.0 in}
{\Large \bf Condensation rank of injective Banach spaces} \\ [0.4in]

{\large Majid Gazor}
\footnote{Fax: (98-311) 3912602;
Email: mgazor@cc.iut.ac.ir}

\vspace{0.15in}
{\small {\em Department of Mathematical Sciences,
Isfahan University of Technology
\\[-0.5ex]
Isfahan, 84156-83111, Iran
}}
%\subjclass[2000]{Primary 46B25; 03E10; 54A05 Secondary 28A05}
%\keywords{Condensation rank; Injective Banach space; Borel derivative.}%
%\date{ \today }%
%\dedicatory{}%
%\commby{}%
% ----------------------------------------------------------------
\end{center}

\noindent
\rule{6.5in}{0.012in}

\vspace{0.1in}
\noindent

The condensation rank associates any topological space with a unique ordinal number. In this paper we prove that the condensation
rank of any infinite dimensional injective Banach space is equal to
or greater than the first uncountable ordinal number.

\vspace{0.10in}

\vspace{0.10in} \noindent

\noindent
\rule{6.5in}{0.012in}

\vspace{0.2in}

\section{ Introduction}

Topological derivative has been used as a tool for the understanding of topological spaces and classification of Banach spaces, see \eg \cite{BessegaPel,baker,bakerTop,PelchSemadani}, for some recent and relevant results see \eg \cite{Galego11AMS,Galego09AMS,GalegoFun,Samuel09}. Topological derivative is a Borel derivative that it derives the limit point derived set of a set. In this paper we define a Borel derivative that it derives the condensation derived set of a set and call condensation derivative. Indeed, one generally expects to have the iterative condensation derived sets of any given set converged much faster than that of the corresponding iterative limit point derived sets of the given set. One, however, should notice that the limit point rank of the real line and also the infinite dimensional Banach space $l_\infty(\NZ)$ is both the first uncountable number $\Omega.$ In other words, the limit point rank can not distinguish between the real line and $l_\infty(\NZ).$ This is, of course, different from condensation rank; $CR(R)=1$ while $CR(l_\infty(\NZ))=\Omega.$ This is our fundamental justification for defining condensation derivative and its associated rank. This paper indeed also raises the possibility of using condensation derivative, and its generalized derivatives associated with arbitrary infinite cardinal numbers, as a tool for better understanding and the classification of classical Banach spaces.

A second motivation for defining condensation derivative and its rank is for its application in obtaining a measure theoretical version of Aleksandrov's Theorem, see \cite{mg,mgph}. In 1916, P. S. Aleksandrov proved that any uncountable Borel set in a separable complete metric space contains a nonempty perfect subset. One year later, M. Ya. Suslin introduced the class of analytic sets for which Aleksandrov's Theorem was readily extended to this class of sets, see e.g. \cite{bbt,Ku68,P}. Since then, perfect sets in complete separable metric spaces have been
studied extensively. One, however, might not find many results in general Hausdorff topological spaces because there are not enough tools to construct Cantor and perfect sets from application of Aleksandrov and Cantor's ideas. In order to make such construction possible we have recently defined the condensation derivative, see \cite{mg}. 
The condensation derivative is a Borel derivative which is measure-preserving on closed, but not $F_\sigma$, subsets of any non-atomic regular Borel measure space. A ``sufficient'' number of iterations of the condensation derivative function composed on a suitable set approaches a perfect set; a sufficient number here refers to a sufficient large ordinal number. The necessary number of iterations depends on the initial set and the topological property
of the whole space. This has led to the notion of the condensation rank of topological spaces, see \cite{mg} and also c.f. \cite{CP} for the classical rank due to Cantor-Bendixson. The measure-preserving property of the condensation derivative and its iteration up to the condensation rank of the space provides a sufficient tool to prove a modified version of the Aleksandrov's theorem: a set with finite and positive measure in any regular non-atomic Borel measure space contains a perfect set whose measure is positive. It is also shown that for any ordinal number, say $\alpha,$ there exists an appropriate totally imperfect Hausdorff
topological space whose condensation rank is $\alpha.$ In the case of non-limit ordinal numbers, the space can be a totally imperfect
``compact" space, see \cite{mg}. Some relevant discussions and applications in locally compact groups are also presented in \cite{mgph}. However, there has not been any discussion on or result of the condensation rank of Banach spaces. 

In the next section we define the condensation derivative and the condensation rank. We also provide some necessary preliminary results from \cite{mg, mgph}. In section \ref{inject}, we present our main result, that is, the condensation rank of any infinite dimensional injective Banach space is equal to or greater than the first uncountable ordinal number.

\section{The condensation rank, derivative and perfect sets}\label{sec2}

Let $X$ be a Hausdorff topological space. A Borel derivative on $2^X$ is a Borel map $D:2^X\rightarrow 2^X$ which is monotone on the
closed subsets of $X$, i.e., $D(K)\subseteq K$ for any closed set $K$. For a Borel derivative $D:2^X\rightarrow 2^X$ and an ordinal
number $\alpha$, the $\alpha$-th iterated derivative $D^\alpha: 2^X\rightarrow 2^X$ is defined inductively as follows: $D^0(K)=K$,
$D^{\alpha+1}(K)=D(D^\alpha(K))$ and $D^\alpha(K)=\bigcap_{\beta\prec\alpha} D^\beta(K)$, for limit ordinal number $\alpha$. Each $D^\alpha$ is a Borel map, c.f \cite{CM83}, where the Borel complexity of the iterations is investigated. An illuminating presentation of Borel derivatives is given by Kechris \cite{Ke94}.

A point $p\in X$ is called a condensation point of $A\subseteq X$ if any neighborhood of $p$ contains an uncountable number of points from $A$. We refer to the set of all condensation points of $A$ as the condensation derived set ($\rm CDS$) of $A$ and denote $\rm CD$ for the condensation derivative, that is a set valued function, which maps any set to its $\rm CDS$.

\begin{lem}\label{ordinal} $\rm CD$ is a Borel derivative and for any set $A$, there exists an ordinal number $\alpha_0$  which, $ {\rm
CD}^\alpha(A)={\rm CD}^{\alpha_0}(A),$ for any $\alpha\succeq \alpha_0.$ In addition, ${\rm CD}^{\alpha_0}(A)$ is a perfect subset
of the perfect kernel of $A$.
\end{lem}

Note that the maximal perfect subset of the closure of a set is called its perfect kernel. Denote $\alpha_A$ as the least ordinal
number among those which satisfy Lemma \ref{ordinal}. Then, $\{{\rm CD}^\alpha(A)\}^{\alpha_A}_{\alpha=1}$ is a strictly descending
chain.

\begin{defn}[The condensation rank]\label{CR} We refer to an ordinal number as the condensation rank, ${\rm CR}(X),$ of the topological
space $X$, when ${\rm CR}(X)=\sup\{\alpha_A|A\subseteq X\}.$ Thus, we have ${\rm CD}^\beta={\rm CD}^{\beta+1}$, in which
$\beta\succeq{\rm CR}(X)$.\end{defn} The idea of the ${\rm CR}$ defined here is similar, but not identical idea, to the idea of
classical rank due Cantor-Bendixson, c.f. \cite{CP}.

\begin{thm} \label{hausloc}
Let $X$ be a Hausdorff topological space. When $X$ is a locally compact space or a complete metric space, any perfect set in $X$ is
the invariant set of ${\rm CD}.$ In addition, if set $A$ is such that $$\emptyset\neq H\subseteq A\cup {\rm CD}(A)\subseteq \overline{A},$$
where $H$ is a perfect set. Then, there is an ordinal number $\alpha_0$ where
$$
H\subseteq P={\rm CD}^{\alpha_0}(A)={\rm CD}^{\alpha}(A), \hbox{where } \alpha\succeq\alpha_0.
$$ In fact, $P$ is the nonempty perfect kernel of $A$ and the invariant set of ${\rm CD}$.
\end{thm}

\begin{rem}\label{codenrem}
Let us now recall the following facts:
\begin{enumerate}
\item{Any non-limit ordinal number $\alpha$ can be associated with an appropriate totally imperfect Hausdorff compact space whose ${\rm
CR}$ is $\alpha,$ see \cite{mg}. Note that the $CR$ of any totally imperfect Hausdorff compact space is not a limit ordinal number.}

\item{In a metric locally compact space, any ${\rm CDS}$ of any uncountable closed set is an invariant perfect set for ${\rm CD}.$
It, however, may be the void set. Indeed, even a non-discrete metric locally compact group space may contain an uncountable totally
imperfect closed set, see \cite{mgph}.}

\item{ Let $\alpha=CR(X), P=CD^{\alpha}(X),$ $X$ a Hausdorff regular space and $X/P$ be the associated quotient topology.
Then, $CD^{\alpha+1}(X/P)=\emptyset$ and $CR(X/P)=\alpha$ or $\alpha+1$. Furthermore, let $\tau_1\subseteq \tau_2$ be totally imperfect Hausdorff topologies on $X.$ Then, $CR^{\tau_1}(X)\geq CR^{\tau_2}(X).$}
\item  Assume $CD^\alpha(A)\supset CD^{\alpha+1}(A)\supset CD^{\alpha+2}(A)$ for an ordinal number $\alpha.$ Then,
$|CD^\alpha(A)\setminus CD^{\alpha+1}(A)|\geq \aleph_1.$
\end{enumerate}
\end{rem}

\begin{prop}\label{rank1} Let $X$ be a discrete, second-countable, separable Banach space or a metric locally compact space. Then, the
condensation rank of $X$ is the first ordinal number.
\end{prop}
\begin{proof}
This is a straightforward result from Theorem \ref{hausloc} and the fact that any metric locally compact space is locally
second-countable.
\end{proof}

\section{The condensation rank of injective Banach spaces}\label{inject}

We first determine the condensation rank of $l_\infty.$ This is indeed our main nontrivial result of this article. 
%Note that the continuum hypothesis is not assumed.
\begin{lem} \label{linf}
The condensation  rank of $l_\infty$ is  the first uncountable ordinal number $\Omega_1$, i.e. ${\rm CR}(l_\infty)=\Omega_1$. In
addition, $$\{{\rm CR}(A)| A\subseteq l_\infty, A\hbox{ is a totally impefect closed set}\}=[1, \Omega_1].$$
\end{lem}

\begin{proof} Let $\aleph_1$ denote the first uncountable cardinal number. It is not too difficult to see that ${\rm CR}(l_\infty)\preceq
\Omega_1,$ since otherwise with Cantor's Theorem and Remark \ref{codenrem} we have
$$Card(l_\infty)\succeq \aleph_1^{\aleph_1},
$$
which is a contradiction. We first apply transfinite induction on ordinal number $\alpha$ in which $\alpha\prec \Omega.$ Then, we separately deal with the case $\alpha=\Omega.$

Although for $\alpha=1$ the proof is trivial, we still need to consider the case $\alpha=2.$ We consider a real number, $a\in \mathbb{R}$, and the natural numbers sequence $\{n\}_{n=1}^\infty=(1,2,3,\cdots)$. Denote the set of all subsequences of $\{n\}_{n=1}^\infty$ with
$\Sigma^2=\Sigma_{\{n\}_{n=1}^\infty}$. For a natural number $m$ and a sequence $\{n_k\}_{k=1}^\infty\in \Sigma^2,$ define

$$a^2_{n}\big(a, m, \{n_k\}_{k=1}^\infty\big)= \left\lbrace
\begin{array}{ll}
a+1/m & \hbox{if}\quad n=n_r\\a&
\hbox{otherwise.}\end{array}\right.$$ Consider
$$\begin{array}{lll}
A_2&=&A_2(a)\\
&=&\Big\{\big\{a^2_n\big(a,
m,\{n_k\}_{k=1}^\infty\big)\big\}_{n=1}^\infty\big| m\in N,
\{n_k\}_{k=1}^\infty\in \Sigma^2\Big\}\\
& &\bigcup \{(a, a, a, \cdots)\},
\end{array}$$
\noindent then $A_2$ is an uncountable closed subset, where $P=\{(a,a,a, \cdots)\}={\rm CD}(A_2)$ and ${\rm CD}^2(A_2)={\rm
CD}(P)=\emptyset,$ i.e. ${\rm CR}(A_2)=2.$ By theorem \ref{hausloc}, $A_2$ is a totally imperfect closed set.

Note that it is easy to modify $A_2(a)$ to obtain the closed set
$$A_2(\{a_n\}_{n=1}^\infty)\subset l_\infty$$ such that
$${\rm CD}(A_2(\{a_n\}_{n=1}^\infty))=\{\{a_n\}_{n=1}^\infty\}, \hbox{ where } \{a_n\}_{n=1}^\infty\in
l_\infty.$$

Now let $\alpha$ be a non-limit ordinal number $\alpha$, in which $$\theta+2=\alpha\prec \Omega.$$ For a sequence $\{b_n\}\in
A_{\theta+1}(a)$, a natural number $N\in \mathbb{N}$ and a sequence $\{n_k\}_{k=1}^\infty\in \Sigma^2,$ we define
$$
\begin{array}{lll}
a_n^\alpha&=&a_n^\alpha \Big(\{b_n\}_{n=1}^\infty, N, \{n_{k}\}_{k=1}^\infty\Big)\\
&=&\left\lbrace
\begin{array}{ll}
b_n+1/N& \hbox{if} \quad n=n_k\\
b_n& \hbox{ otherwise}
\end{array}
\right.
\end{array}
$$
\noindent and
$$
A_\alpha(a)=\Big\{\{a^\alpha_n\big(\{b_n\}, N,
\{n_k\}\big)\}\big|\{b_n\}_{n=1}^\infty\in A_{\theta+1}(a) , N\in
\mathbb{N}, \{n_k\}\in \hbox{$\Sigma^2$}\Big\}.
$$

\noindent Then, we have ${\rm CD}(A_\alpha(a))=A_{\theta+1}$ and ${\rm CD}^{\theta+1}(A_\alpha(a))=\{(a, a, a, \cdots)\},$ i.e. ${\rm
CD}(A_\alpha(a))=\theta+2.$

Now consider a non-limit ordinal number $\alpha$ such that $$\alpha=\theta+1\prec \Omega,$$ where $\theta$ is a limit ordinal
number. Let the set $\{\theta_n|n\in \mathbb{N}\}$ be the set of all (possibly rearranged) non-limit predecessor ordinal numbers of $\theta$. Thus, if

$$
A_\alpha(a)=\bigcup_{n=1}^\infty A_{\theta_n}(a+1/n),
$$

\noindent then we have ${\rm CD}^\theta(A_\alpha(a))=\{(a, a, a, \cdots)\}$ and ${\rm CR}(A_\alpha(a))=\theta+1.$

For a limit ordinal number $\alpha$, we define $$A_\alpha(a)=\bigcup_{n=1}^\infty A_{\alpha_n}(a+n)\cup \{(a, a, a, \cdots)\},$$
\noindent in which $\{\alpha_n|n\in \mathbb{N}\}$ is the set of (possibly rearranged) PONs of $\alpha$. Hence,
$${\rm CD}^{\alpha_n}(A_\alpha(a))\neq \emptyset \hbox{ and } {\rm CD}^\alpha(A_\alpha(a))=\emptyset.$$
Therefore, transfinite induction is complete. It is easy to modify the above to obtain $A_\alpha(\{a_n\}^\infty_{n=1})$ for any ordinal number $\alpha\prec \Omega_1$ and $\{a_n\}^\infty_{n=1}\in l_\infty,$ such that ${\rm CD}^\alpha(A_{\alpha+1}(\{a_n\}^\infty_{n=1}))=\{\{a_n\}^\infty_{n=1}\}.$

Finally, for constructing $A_{\Omega_1}$, let $A_1=\big\{(a_n)_{n=1}^\infty| a_n\in \{0,1\}\big\}$ and
$$
\Psi : [1, \Omega_1) \rightarrow \hbox{$A_1$}
$$
be a bijective function. Then, define
$$A_{\Omega_1}=\bigcup_{\alpha\prec \Omega_1}
\Big\{A_\alpha\big(\Psi(\alpha)\big)\Big\},$$

\noindent and thereby we have
$${\rm CD}^\alpha(A_{\Omega_1})\neq
\emptyset, \hbox{ where } \alpha\prec \Omega_1, \hbox{ and } {\rm
CD}^{\Omega_1}(A_{\Omega_1})=\emptyset.$$

\noindent It can be seen that $A_{\Omega_1}$ is a totally imperfect closed set and ${\rm CR}(A_{\Omega_1})=\Omega_1.$ This completes the
proof.
\end{proof}

Let us now recall a well known Theorem from \cite{linds}.

\begin{thm}\cite[Theorem 2.f.3]{linds}\label{injectlinf} Every infinite-dimensional injective  Banach space has a subspace isomorphic to $l_\infty.$
\end{thm}

Note that Theorem \ref{injectlinf} implies that such a subspace is a completed ``closed" subspace isomorphic to $l_\infty.$ The following
is our main result in this paper.
\begin{thm}
Let $X$ be an infinite dimensional injective Banach space. Then, ${\rm CR}(X)\succeq \Omega_1,$ where $\Omega_1$ is the first
uncountable ordinal number. In addition, $$\{{\rm CR}(A)| A\subseteq
X, A\hbox{ is a totally impefect closed set}\}\supseteq[1,
\Omega_1].$$
\end{thm}
\begin{proof} This is straight forward based on Theorem \ref{injectlinf} and Lemma \ref{linf}.
\end{proof}

Note the fact that the limit point rank is unable to properly distinguish between the real line and infinite dimensional injective Banach spaces, it signifies the necessity to define the condensation rank which delicately addresses this issue. Therefore, this may also imply the necessity to define condensation derivatives according to arbitrary cardinals.

\begin{prob}
For any infinite cardinal number $\aleph$ we can define $\aleph$-condensation point. Accordingly, we can define $\aleph$-\CR$(X).$ Let $\aleph=2^{|\Gamma|}$ and $\Omega_\aleph$ denote the first ordinal with having $\aleph$ number of predecessors. Then, is it correct to say that $\aleph-\CR(l_\infty(\Gamma))= \Omega_\aleph$?
\end{prob}

\section{An invitation message}

It is my opinion that this draft still needs some works for a descent publication. However, I have not continued to work on this subject since September 2005. This is mainly because my research area and expertise is now very far from this subject. I do not have a concrete plan to work on this subject in a foreseeable future. Thus, I hereby invite any potential reader of this article for his or her feedback and critical comments as well as any possible joint collaboration on improvement of these results for a joint publication.

\end{document}